\documentclass[a4paper]{scrartcl}
\usepackage[utf8]{inputenc}
\usepackage[T1]{fontenc}
\usepackage[UKenglish]{babel}

\usepackage{amsmath}
\usepackage{amssymb}
\usepackage{amsthm}
\usepackage{authblk}

\usepackage{hyperref}

\newtheorem{lemma}{Lemma}[section]
\newtheorem{proposition}[lemma]{Proposition}
\newtheorem{corollary}[lemma]{Corollary}
\newtheorem{theorem}[lemma]{Theorem}
\theoremstyle{remark}
\newtheorem{remark}[lemma]{Remark}
\newtheorem{example}[lemma]{Example}
\theoremstyle{definition}
\newtheorem{definition}[lemma]{Definition}

\DeclareMathOperator{\hav}{hav}
\DeclareMathOperator{\lip}{Lip}

\title{On the existence of fixed points for typical nonexpansive mappings on spaces with positive curvature}

\author[1]{Christian~Bargetz}
\author[1]{Michael~Dymond}
\author[1]{Emir~Medjic}
\author[2]{Simeon~Reich}
\affil[1]{University of Innsbruck}
\affil[2]{The Technion---Israel Institute of Technology}

\date{15th June 2020}

\begin{document}

\maketitle

\begin{abstract}
  We show that the typical nonexpansive mapping on a small enough subset of a CAT($\kappa$)-space is a \textit{contraction in the sense of Rakotch}. By typical we mean that the set of nonexpansive mapppings without this property is a $\sigma$-porous set and therefore also of the first Baire category. Moreover, we exhibit metric spaces where strict contractions are not dense in the space of nonexpansive mappings. In some of these cases we show that all continuous self-mappings have a fixed point nevertheless.
\end{abstract}

\section{Introduction}

The celebrated fixed point theorem of Brouwer asserts that every continuous self-mapping of a bounded, closed 
and convex subset of a Euclidean space has a fixed point. Unfortunately, this result does not generalise to 
bounded, closed and convex subsets of infinite-dimensional Banach spaces, as in this case 
the unit sphere is a continuous retract of the closed unit ball by Dugundji's extension theorem \cite{Dugundji}. 
As a matter of fact, the unit sphere is known to be even a Lipschitz retract of the closed unit ball in all 
infinite-dimensional Banach spaces \cite{BS}. 
For this reason, in the infinite-dimensional setting it is more interesting to only consider the nonexpansive mappings
instead of the continuous ones. Recall that a mapping
\[
  f\colon X \to X
\]
on a metric space $(X,\rho)$ is called \emph{nonexpansive} if it satisfies
\[
  \rho(f(x),f(y)) \leq \rho(x,y)
\]
for all $x,y\in X$. In~\cite{Browder} F.~E.~Browder showed that every nonexpansive self-mapping 
of the closed unit ball of the Hilbert space $\ell_2$ has a fixed point. This result was later extended to 
nonexpansive self-mappings of bounded, closed and convex subsets of all uniformly convex spaces Banach spaces and beyond; see, for example,~\cite{GK1990, GR1984, Piatek}.

In~\cite{BM1976, BM1989} F.~S.~de~Blasi and J.~Myjak took a different approach to this problem: 
instead of aiming at a characterisation of the Banach spaces with the property that every nonexpansive 
self-mapping of a bounded, closed and convex subset has a fixed point, they showed that the typical 
nonexpansive self-mapping of a bounded, closed and convex subset of an \emph{arbitrary} Banach space has a 
fixed point. In this case, the term ``typical'' means that the set of mappings without this property is 
very small in the topological sense. In the first of these articles, it is shown that this set is the complement of a
dense $G_\delta$-set and in the second one that it is even $\sigma$-porous. 
Roughly speaking, the concept of porous sets is a quantitative version of the concept of nowhere dense sets, 
which was introduced by A.~Denjoy in~\cite{Denjoy}. For a detailed discussion of different concepts of porosity, 
we refer the interested reader to~\cite{Zajicek}.

It seems natural to ask whether it is possible to explain the existence of fixed points of 
typical nonexpansive mappings in terms of fixed point theorems. Already in their original papers, 
de Blasi and Myjak showed that in the case of Hilbert spaces, it is not Banach's fixed point theorem which is behind 
their result. More precisely, they showed that the set of strict contractions, that is, the set of nonexpansive mappings 
$f \colon C\to C$ for which there is an $L<1$ satisfying
\[
  \rho(f(x),f(y)) \leq L \rho(x,y)
\]
for all $x,y\in C$ is not only of the first Baire category, but even $\sigma$-porous. In~\cite{BD2016} this 
result was generalised to general Banach spaces and in~\cite{BDR2017} to a class of metric spaces.

In~\cite{Rakotch}, E.~Rakotch proved a generalisation of Banach's fixed point theorem. More precisely, 
in this more general version the Lipschitz constant is replaced by a decreasing function. 
A mapping $f\colon C\to C$ is called \emph{contractive in the sense of Rakotch} if there is a decreasing function 
$\phi_f\colon[0,\mathrm{diam}(C)]\to [0,1]$ with $\phi_f(t)<1$ for $t>0$ and such that
\[
  \rho(f(x),f(y)) \leq \phi_f(\rho(x,y)) \rho(x,y)
\]
for all $x,y\in C$. In~\cite{RZ2001} the fourth author together with A.~J.~Zaslavski exhibit a set of nonexpansive mappings 
which are all contractive in the sense of Raktoch and the complement of which is $\sigma$-porous. In some sense, this can 
be seen as an explanation of the results of de Blasi and Myjak in terms of fixed point theorems. 
In~\cite{RZ2016:Two} this result was generalised to closed star-shaped subsets of hyperbolic spaces. The class of 
hyperbolic spaces was introduced in~\cite{RZ1990} and it contains those metric spaces with a big enough family of metric 
lines where the resulting triangles are thinner than in Euclidean space.

With the aim of developing a curvature theory for spaces without differentiability structure  
M.~Gromov introduced in~\cite{Gromov1987} the concept of $\mathrm{CAT}(\kappa)$-spaces. 
Roughly speaking, a $\mathrm{CAT}(\kappa)$-space is a geodesic metric space where the triangles are thinner than the ones 
in a corresponding model space. These model spaces are either Euclidean space or a suitably scaled version of the sphere 
or of the classical 2-dimensional hyperbolic space in the sense of non-Euclidean geometry. 
For more details on $\mathrm{CAT}(\kappa)$-spaces, we refer the reader to~\cite{BH1999, Burago}. 
In~\cite{BDR2017} the concept of weakly hyperbolic spaces was introduced. 
The class of weakly hyperbolic spaces contains the class of hyperbolic spaces 
and the class of all $\mathrm{CAT}(\kappa)$-spaces for all $\kappa\in\mathbb{R}$. 
In~\cite{BDR2017} it is shown that the typical nonexpansive mapping on a closed and star-shaped subset 
of a weakly hyperbolic space has the maximal possible Lipschitz constant one. 
Combining this result with those in~\cite{RZ2016:Two} we see that, in the case of hyperbolic spaces, 
the typical nonexpansive mapping is contractive in the sense of Rakotch, but nevertheless has Lipschitz constant one. 
One could interpret this statement in the sense that, in the case of non-positive curvature, 
on large scales the typical nonexpansive mapping behaves like a contraction but on small scales it approximates an isometry. 
Note that the result on the Lipschitz constant of the typical nonexpansive self-mapping is also known  
in the case of positive curvature, whereas the question of whether it is contractive in the sense of Rakotch has not yet been resolved. 
The aim of the present note is to address this question for the class of $\mathrm{CAT}(\kappa)$ with $\kappa >0$.

For recent results in metric fixed point theory on spaces with positive curvature, 
we refer the reader to~\cite{MR2508878, ezawa2020convergence, MR2847429}.

Since in the case of positive curvature, we need a bound on the diameter of the convex set, 
we only consider spaces of nonexpansive mappings on bounded, closed and convex subsets 
of $\mathrm{CAT}(\kappa)$-spaces.

\section{Preliminaries and notation}
Recall that a metric space $(X,\rho)$ is called \emph{geodesic} if for every pair of points $x,y\in X$, there is an 
isometric embedding $c\colon [0,\rho(x,y)]\to X$ with $c(0)=x$ and $c(\rho(x,y))=y$. We call the image of this interval 
a \emph{metric segment joining $x$ and $y$} and denote it by $[x,y]$. Note, however, that in general the above metric 
segment might not be unique so, in general, the notation $[x,y]$ is not well defined. Given $x,y\in X$ with a unique metric 
segment $[x,y]$, for $\lambda\in[0,1]$, we may define $(1-\lambda)x\oplus \lambda y$ as the unique point $z$ 
on this segment with
\[
  \rho(x,z) = \lambda \rho(x,y)\qquad\text{and}\qquad \rho(y,z) = (1-\lambda) \rho(x,y).
\]
An isometric image of the whole real line $\mathbb{R}$ is referred to as a metric line. We call a subset $C\subset X$
\emph{(metrically) convex} if for every pair of points $x,y\in C$ it contains a metric segment connecting them. 

Next, we recall the definition of $\mathrm{CAT}(\kappa)$-spaces from~\cite{BH1999}; see~Definition~2.6 
in~\cite[p.~96]{BDR2017}.

\begin{definition}
  \begin{enumerate}
  \item We define a family of model spaces $(M_{\kappa})$, where $\kappa\in\mathbb{R}$, as follows:
    \begin{enumerate}
    \item For $\kappa>0$ we let $M_{\kappa}$ denote the metric space given by the two-dimensional sphere $\mathbb{S}^{2}$ with its intrinsic metric, scaled by a factor of $1/\sqrt{\kappa}$.
    \item We define $M_{0}$ as the Euclidean space $\mathbb{R}^{2}$.
    \item For $\kappa<0$ we write $M_{\kappa}$ for the hyperbolic space 
     $\mathbb{H}^{2}$ (see~\cite[Definition~2.10]{BH1999}) with the metric scaled by a factor of $1/\sqrt{-\kappa}$. 
    \end{enumerate}
    We write $d_{\kappa}$ for the metric on $M_{\kappa}$.
  \item Let $\kappa\in \mathbb{R}$ and $(X,\rho_{X})$ be a metric space. 
    Given three points $x_{1},x_{2},x_{3}\in X$ and metric segments of the form 
    $[x_{1},x_{2}],[x_{2},x_{3}],[x_{3},x_{1}]\subseteq X$, we call the union 
    $[x_{1},x_{2}]\cup[x_{2},x_{3}]\cup[x_{3},x_{1}]$ a \emph{geodesic triangle} with vertices $x_{1},x_{2},x_{3}$.
    A geodesic triangle with vertices $\overline{x}_{1},\overline{x}_{2}, \overline{x}_{3}$ in $M_{\kappa}$ 
    is said to be a \emph{comparison triangle} for a geodesic triangle with vertices $x_{1},x_{2},x_{3}$ in $X$ if 
    $d_{\kappa}(\overline{x}_{i},\overline{x}_{j})=\rho_{X}(x_{i},x_{j})$. 
    A point $\overline{x}\in [\overline{x}_{i},\overline{x}_{j}]$ is called a \emph{comparison point} 
    for $x\in[x_{i},x_{j}]$ if $d_{\kappa}(\overline{x},\overline{x}_{k})=\rho_{X}(x,x_{k})$ for $k=i,j$.
    \item Let $(X,\rho_{X})$ be a metric space. If $\kappa\leq 0$, then $(X,\rho_{X})$ is called a 
    $\mathrm{CAT}(\kappa)$-space if it is geodesic and every geodesic triangle $T$ in $X$ has a comparison triangle 
    $\overline{T}$ in $M_{\kappa}$ such that
    \begin{equation}\label{eq:comparisontriangle}
      \rho_{X}(x,y)\leq d_{\kappa}(\overline{x},\overline{y})
    \end{equation}
    whenever $\overline{x},\overline{y}\in\overline{T}$ are comparison points for $x,y\in T$. 
If $\kappa>0$, then we define a constant $D_{\kappa}=\mathrm{diam} M_{\kappa}=\frac{\pi}{\sqrt{\kappa}}$ 
and we say that $(X,\rho_{X})$ is a $\mathrm{CAT}(\kappa)$ space if for every pair of points 
$x,y\in X$ with $\rho_X(x,y)< D_\kappa$ there is a metric segment joining $x$ and $y$  and~\eqref{eq:comparisontriangle} 
is satisfied for all geodesic triangles $T\subseteq X$ with perimeter smaller than $2D_\kappa$, that is, 
$\rho_X(x,y)+\rho_X(y,z)+\rho_X(z,x)< 2D_{\kappa}$, where $x,y,z$ denote the vertices of $T$.
  \end{enumerate}
\end{definition}
Thus, $\mathrm{CAT}(\kappa)$ spaces can be thought of as metric spaces for which 
every sufficiently small geodesic triangle is `thinner' in all directions than a corresponding triangle 
in the model space $M_{\kappa}$. The classes of $\mathrm{CAT}(\kappa)$ spaces are increasing in the sense 
that whenever $X$ is a $\mathrm{CAT}(\kappa)$ space, it is also a $\operatorname{CAT}(\kappa')$ space 
for all $\kappa'\geq \kappa$; see~\cite[Theorem~1.12]{BH1999}.

Finally, we recall the definition of $\sigma$-porous sets.
\begin{definition}
  Given a metric space $(M,d)$, a subset $A\subset M$ is called \emph{porous at a point $x\in A$} 
if there exist $\varepsilon_0>0$ and $\alpha>0$ such that for all $\varepsilon\in(0,\varepsilon_0)$, 
there exists a point $y\in B(x,\varepsilon)$ such that $B(y,\alpha\varepsilon)\cap A = \emptyset$. 
The set $A$ is called \emph{porous} if it is porous at all its points and $A$ is called \emph{$\sigma$-porous} 
if it is the countable union of porous sets.
\end{definition}
Note that the definition of porous sets differs slightly from the one used in~\cite{BM1989}. 
For $\sigma$-porous sets both definitions agree. A detailed discussion of these differences 
and why for $\sigma$-porosity they do not matter can be found in~\cite[p.~93]{BDR2017}.

\section{The typical nonexpansive self-mapping on a small enough set is a Rakotch contraction}

Let $(X,\rho)$ be a $\mathrm{CAT}(\kappa)$ space and let $C\subset X$ be a bounded and closed subset. 
We consider the space
\begin{equation}\label{eq:space_non_exp_mappings}
  \mathcal{M} := \{f\colon C\to C\colon \lip f \leq 1\}
\end{equation}
equipped with the metric
\[
  d_{\infty}(f,g) := \sup\{\rho(f(x),g(x))\colon x\in C\}.
\]
We show that on small enough closed and suitably ``star-shaped'' subsets of a $\mathrm{CAT}(\kappa)$ space the typical nonexpansive 
self-mapping is a Rakotch contraction. In other words, we intend to show the following theorem:

\begin{theorem}\label{thm:typicalRakotch}
  Let $(X,\rho_{X})$ be a $\mathrm{CAT}(\kappa)$ space with $\kappa\in\mathbb{R}$,
  \begin{equation*}
    D_{\kappa}:=\begin{cases}
      \infty & \text{ if }\kappa\leq 0,\\
      \frac{\pi}{\sqrt{\kappa}} & \text{ if }\kappa >0,
    \end{cases}
  \end{equation*} 
  $C\subset X$ a bounded closed set, and $x\in C$ such that $C\subset B(x,D_{\kappa}/2)$ and for every $z\in C$ the metric segement $[z,x]$ connecting $z$ to $x$ is contained in $C$. Let $\mathcal{M}$ denote the space of nonexpansive mappings on $C$ defined by \eqref{eq:space_non_exp_mappings}. Then the set of mappings which are not contractive in the sense of Rakotch is a $\sigma$-porous subset of~$\mathcal{M}$.
\end{theorem}

\begin{remark}
  Combining the above theorem with Theorem~3.3 of~\cite[p.~101]{BDR2017}, we obtain that under the above assumptions on $C$, the set of nonexpansive self-mappings of~$C$ which are contractive in the sense of Rakotch and have Lipschitz constant one is the complement of a $\sigma$-porous set. In other words, the typical nonexpansive self-mapping is both contractive in the sense of Rakotch and has the maximal possible Lipschitz constant one.
\end{remark}
\subsection*{Proof of Theorem~\ref{thm:typicalRakotch}.}
Roughly speaking, the strucutre of the proof of Theorem~\ref{thm:typicalRakotch} can be described as follows: We start by establishing some tools on the two-dimensional sphere~$\mathbb{S}^{2}$. Using these tools, we show that the set of strict contractions is a dense subset of~$\mathcal{M}$. We pick a null sequence $(a_n)$ by setting $a_n=R/n$. Then using the density of the strict contractions, given $f\in\mathcal{M}$, we can find a strict contraction $f_\gamma$ arbitrarily close. Then we show that all mappings $g\in\mathcal{M}$ which are close to $f_\gamma$, on distances larger than $a_n$ behave like strict contractions. Since $a_n\to 0$, in the last step we may conclude that the typical $f\in\mathcal{M}$ is a Rakotch contraction.

As outlined above, we first consider the case of the sphere $\mathbb{S}^{2}$ in the three-dimensional Euclidean space. In this case we have $D_\kappa = \pi$ and hence we let $0<R<\frac{\pi}{2}$ and $x\in\mathbb{S}^{2}$ be given. We denote by $\rho_{\mathbb{S}^{2}}$ the intrinsic metric on the sphere and set
\[
  B(x,R):=\{y\in\mathbb{S}^{2}\colon \rho_{\mathbb{S}^{2}}(x,y) < R\}.
\]
Now for $t\in [0,1]$, we define the mapping
\[
  s_{x,R,t}\colon B(x,R)\to B(x,R), \qquad z \mapsto t z \oplus (1-t) x
\]
and note that
\[
  \rho_{\mathbb{S}^{2}}(x,s_{x,R,t}(y)) = t \rho_{\mathbb{S}^{2}}(x,y) < tR
\]
and
\[
  \rho_{\mathbb{S}^{2}}(y,s_{x,R,t}(y)) = (1-t)  \rho_{\mathbb{S}^{2}}(x,y) < (1-t)R.
\]
In order to show that Banach contractions form a dense subset of~$\mathcal{M}$, we first investigate the Lipschitz constant 
of~$s_{x,R,t}$.

\begin{lemma}\label{lem:Contr}
  Let $t\in [0,1)$ be given. The function $s_{x,R,t}$ is a strict contraction. More precisely, we have
  \[
    \lip s_t\leq t^{1/k_R}
  \]
  for some $k_R\geq 2$ which only depends on $R$.
\end{lemma}

\begin{proof}
  Given two points $y,z\in B(x,R)$, we use the notation
  \[
    a := \rho_{\mathbb{S}^{2}}(x,y), \qquad b := \rho_{\mathbb{S}^{2}}(x,z) \qquad\text{and}\qquad c := \rho_{\mathbb{S}^{2}}(y,z).
  \]
  In order to abbreviate the notation, we write $s_t$ for the function $s_{x,R,t}$. Moreover, we set
  \[
    c_t := \rho_{\mathbb{S}^{2}}(s_t(y),s_t(z))
  \]
  and observe that
  \[
    \rho_{\mathbb{S}^{2}}(x, s_t(y)) = ta, \qquad \text{and}\qquad \rho_{\mathbb{S}^{2}}(x,s_t(z)) = tb.
  \]
  Using this notation, our aim is to show that
  \[
    c_t \leq t^{1/k_R} c.
  \]  
  Without loss of generality, we may assume that $a\geq b > 0$. The law of haversines allows us to write
  \[
    \hav c = \hav(a-b) + \sin a\sin b \hav \gamma,
  \]
  where $\gamma$ is the angle of the triangle spanned by $x,y, z$ at the vertex $x$. 
  Again using the law of haversines, we obtain
  \begin{align*}
    \hav c_t & = \hav(t(a-b)) + \sin(ta)\sin(tb) \hav(\gamma) \leq t \hav(a-b) + t^{2/k} \sin a \sin b \hav \gamma \\
             & \leq t^{2/k} (\hav (a-b) + \sin a \sin b \hav \gamma) = t^{2/k} \hav c
  \end{align*}
  for some $k\geq 2$ by Lemma~\ref{lem:PowSinConv}, taking into account that $\hav \vartheta = \sin^2\frac{\vartheta}{2}$ and 
$a-b\in[0,\pi/2)$. Using the concavity of the function $\xi \mapsto \sqrt{\hav \xi}$ on $[0,\pi]$, we may conclude from the above 
that
  \[
    \hav c_t \leq t^{2/k} \hav c\leq \hav (t^{1/k} c)
  \]
  and hence
  \[
    \rho(s_t(y),s_t(z)) = c_t \leq t^{1/k} c = t^{1/k} \rho(y,z),
  \]
  that is, $s_t$ is $t^{1/k}$-Lipschitz. In particular, $s_t$ is a strict contraction.
\end{proof}

\begin{lemma}\label{lem:PowSinConv}
  Let $R\in[0,\pi/2)$. Then there is an integer $k=k_R\in\mathbb{N}$ such that the function
  \[
    [0,R]\to[0,1], \qquad \tau \mapsto \sin^{k} \tau,
  \]
  is convex. In particular, for this $k$ we have
  \[
    \sin(d t) \leq t^{1/k} \sin d
  \]
  for every $d\in [0,R]$ and every $t\in[0,1]$
\end{lemma}

\begin{proof}
  We consider the second derivative of this function and note that
  \[
    k \sin^k(\tau) ((k - 1) \cot^2(\tau) - 1) \geq 0
  \]
  provided that  
  \[
    k \geq 1 + \frac{1}{\cot^2 R} \geq 1 + \frac{1}{\cot^2 \tau}.
  \]
  This is because the cotangent function is decreasing on $(0,R)$.
\end{proof}

Now that on the sphere we have all tools available, we return to the $\mathrm{CAT}(\kappa)$ space~$(X, \rho_X)$.

\begin{lemma}\label{lem:Contr2}
  Let $(X,\rho_X)$ be a $\mathrm{CAT}(\kappa)$ space for some $\kappa\in\mathbb{R}$. Moreover, let $x\in X$ and $R>0$. In the 
case where $\kappa>0$, we assume that $R<\frac{D_\kappa}{2} = \frac{\pi}{2\sqrt{\kappa}}$. For $t\in [0,1]$, the mapping
  \[
    s_{X,x,R,t}\colon B(x,R)\to B(x,R), \qquad z \mapsto t z \oplus (1-t) x
  \]
  is a strict contraction with $\lip s_{X,x,R,t} \leq t^{1/{k_{\kappa,R}}}$ for some $k_{\kappa,R}>0$ which only depends on 
$\kappa$ and $R$.
\end{lemma}

\begin{proof}
  For $\kappa\leq 0$ we may obviously pick $k_{\kappa,R}=1$. Therefore we now consider the case where $\kappa>0$. Let $y,z\in B(x,R)$ 
be given. We consider the triangle with vertices $x,y,z$ and pick a comparison triangle in $M_\kappa$ with vertices $\bar{x},\bar{y},\bar{z}$. In other words, we have points on the unit sphere~$\mathbb{S}^{2}$ with the properties that
  \[
    \rho_{\mathbb{S}^{2}}(\bar{x},\bar{y}) = \sqrt{\kappa} \rho_{X}(x,y), \qquad \rho_{\mathbb{S}^{2}}(\bar{x},\bar{z}) = \sqrt{\kappa} \rho_{X}(x,z)
  \]
  and
  \[
    \rho_{\mathbb{S}^{2}}(\bar{y},\bar{z}) = \sqrt{\kappa} \rho_{X}(y,z).
  \]
  We set $s_t := s_{\bar{x},\sqrt{\kappa}R,t}$. Using the above identities, we may use Lemma~\ref{lem:Contr} to obtain a constant $k:=k_{\sqrt{\kappa}R}$ and conclude that
  \[
    \rho_{X}(s_{X,x,R,t}(y),s_{X,x,R,t}(z)) \leq \frac{\rho_{\mathbb{S}^{2}}(s_{t}(\bar{y}),s_t(\bar{z}))}{\sqrt{\kappa}} \leq t^{1/k}  \frac{\rho_{\mathbb{S}^{2}}(\bar{y},\bar{z})}{\sqrt{\kappa}} = t^{1/k} \rho_{X}(y,z)
  \]
  because $s_{t}(\bar{y})$ is a comparison point for $s_{X,x,R,t}(y)\in [x,y]$ and $s_{t}(\bar{z})$ is a comparison point for 
$s_{X,x,R,t}(y)\in [x,z]$.
\end{proof}

\begin{proof}[{of Theorem~\ref{thm:typicalRakotch}}]
  Since $C$ is bounded, closed and $C\subset B(x,D_\kappa/2)$ for some $x\in C$, we may pick $0<R<D_\kappa/2$ with $C\subset B(x,R)$.
  We set
  \[
    \mathcal{K}_{n} := \left\{f\in\mathcal{M}\colon \exists \beta\in(0,1)\;\text{s.t.}\; \rho(f(y),f(z)) \leq \beta \rho(y,z)\;\text{for}\;\rho(y,z)\geq \frac{R}{n}\right\}
  \]
  and show that $\mathcal{M}\setminus\mathcal{K}_n$ is porous. In order to keep the notation succinct, we use the 
abbreviations $s_t:=s_{X,x,R,t}$ and $k:=k_{\kappa, R}$, where the latter number is the reciprocal of the exponent from 
Lemma~\ref{lem:Contr2} corresponding to~$R$ and~$\kappa$.

  Let $f\in\mathcal{M}$ and $r\in(0,1]$. We set
  \[
    \gamma := \frac{r}{2(R+1)} \qquad\text{and}\qquad  \alpha := \frac{\min\{R,1\}(1-(1-\gamma)^{1/k})}{4n}, 
  \]
  and also $f_\gamma := s_{1-\gamma}\circ f$, that is, 
  \[
    f_\gamma (z)  = (1-\gamma) f(z) \oplus \gamma x
  \]
  for $z\in C$. Then
  \[
    d_\infty(f,f_\gamma) \leq \gamma R < r
  \]
  and
  \[
    \rho(f_\gamma(y),f_\gamma(z)) \leq (1-\gamma)^{1/k} \rho(y,z)
  \]
  for all $y,z\in C$ by Lemma~\ref{lem:Contr2}. Observe that for $y,z\in C$ with $\rho(y,z) \geq \frac{R}{n}$, we have
  \[
    \rho(y,z) - \rho(f_\gamma(y),f_\gamma(z)) \geq \rho(y,z) - (1-\gamma)^{1/k} \rho(y,z).
  \]
  Now let $g\in B(f_\gamma, \alpha r)$ and observe that
  \[
    \rho(g(y),g(z)) \leq \rho(f_\gamma(y),f_\gamma(z)) +2 \alpha r
  \]
  and
  \begin{align*}
    \rho(y,z)-\rho(g(y),g(z)) & \geq \rho(y,z) - \rho(f_\gamma(y),f_\gamma(z)) -2\alpha r \geq \\
                              & \geq (1-(1-\gamma)^{1/k})\rho(y,z) - 2\alpha r.
  \end{align*}
  Therefore
  \begin{align*}
    \rho(g(y), g(z)) & \leq \rho(y,z) - (1-(1-\gamma)^{1/k})\rho(y,z) + 2\alpha r\\
                     & \leq \rho(y,z) \left(1-(1-(1-\gamma)^{1/k})+\frac{2n\alpha r}{R}\right)\\
                     & \leq  \frac{1+(1-\gamma)^{1/k}}{2}\rho(y,z)
  \end{align*}
  and hence $g\in \mathcal{K}_{n}$. Moreover, using the inequality $d_\infty(f_\gamma, g)\leq \alpha r$, we see that 
  \[
    d_\infty(f,g) \leq d_\infty(f_\gamma,g) + d_\infty(f,f_\gamma) \leq \alpha r + \gamma R \leq \frac{r}{2}+\frac{r}{2}\leq r
  \]
  and hence
  \[
    \{g\in\mathcal{M}\colon d_\infty(g,f_\gamma)\leq\alpha r\}\subset \{g\in\mathcal{M}\colon d_\infty(g,f) \leq r\}.
  \]
  This shows that $\mathcal{M}\setminus\mathcal{K}_n$ is porous. Finally, observe that every mapping 
$g\in\bigcap_{n=1}^{\infty} \mathcal{K}_n$ is contractive in the sense of Rakotch.
\end{proof}
\section{An application}
Whilst Theorem~\ref{thm:typicalRakotch} establishes that the typical mapping $f\in\mathcal{M}$ in the given setting is a Rakotch contraction, it is known, in more general settings, that the typical mapping $f\in \mathcal{M}$ is not a strict contraction, that is, it has Lipschitz constant one; see \cite{BD2016} and \cite{BDR2017}. Therefore, for the typical mapping $f\in\mathcal{M}$, one can ask about the size of the set of points $x\in\mathcal{M}$ which `witness' the fact that $\lip f =1$. There are two natural means to make this question precise: For a mapping $f\in\mathcal{M}$ and point $x\in C$, we consider the quantities
\begin{equation*}
\lip(f,x)=\limsup_{r\searrow 0}\sup_{y\in (B(x,r)\cap C)\setminus\left\{x\right\}}\frac{\rho(f(y),f(x))}{\rho(y,x)}, 
\end{equation*}
and
\begin{equation*}
\widehat{\lip}(f,x):=\sup_{y\in C\setminus\left\{x\right\}}\frac{\rho(f(y),f(x))}{\rho(y,x)}.
\end{equation*}
Observe that $\lip(f,x) \leq \widehat{\lip}(f,x)\leq \lip f$. Moreover, it is easy to come up with examples where the two inequalities in this sequence are strict. Thus, the three compared quantities are, generally speaking, only loosely related. In general, we have the identity $\lip f=\sup_{x\in C}\widehat{\lip}(f,x)$ and in the case where $C$ is convex, we have $\lip f=\sup_{x\in C}\lip(f,x)$; see \cite[Lemma~3.4]{BDR2017}. 

For the typical $f\in\mathcal{M}$, the quantities defined above offer two means to make precise the notion that a point $x\in C$ witnesses the fact that $\lip f=1$. We can either regard $x$ as such a witness if $\lip(f,x)=1$ or we can regard $x$ as such a witness if $\widehat{\lip}(f,x)=1$. Thus, for a mapping $f\in\mathcal{M}$ we consider the following two sets:
\begin{equation*}
R(f):=\left\{x\in C\colon \lip(f,x)=1\right\}
\end{equation*}
and
\begin{equation*}
\widehat{R}(f):=\left\{x\in C\colon \widehat{\lip}(f,x)=1\right\},
\end{equation*}
as defined in \cite{BDR2017}. Since $\lip(f,x)\leq \widehat{\lip}(f,x)$, the latter notion of being a witness is weaker and we have $R(f)\subseteq \widehat{R}(f)$. 

It is easy to come up with examples of settings $X$, $C\subseteq X$ and $\mathcal{M}$, and mappings $f\in \mathcal{M}$ which have the maximal possible Lipschitz constant $\lip f=1$, but behave as a constant mapping on large open subsets of $C$ and therefore satisfy $\lip(f,x)=0$ for a large set of points $x\in C$. Such examples indicate that for mappings $f\in\mathcal{M}$ with $\lip f=1$, the set $R(f)$ need not be large in any sense.

However, when $X$ is a hyperbolic metric space and $C\subset X$ is star-shaped, the paper \cite[Corollary~3.9]{BDR2017} establishes that for the typical $f\in\mathcal{M}$, the set $R(f)$ is large in the sense of the Baire Category Theorem, that is, it is a residual subset of $C$. This result leaves open the case where $X$ is a $\mathrm{CAT}(\kappa)$ space with $\kappa>0$. As a consequence of Theorem~\ref{thm:typicalRakotch}, we are able to treat this case, provided that $C$ is sufficiently small. The proof exploits the following noteworthy property of Rakotch contractions observed without proof in \cite{BD2016}. We provide its short proof in full. 
\begin{lemma}\label{lemma:R=hat(R)}
	Let $(X,\rho_{X})$ be a metric space, $C\subseteq X$ be a closed set and $f\colon C\to C$ be a nonexpansive mapping which is contractive in the sense of Rakotch. Then
	\begin{equation*}
	R(f)=\widehat{R}(f).
	\end{equation*}
\end{lemma} 
\begin{proof}
  We only need to verify the inclusion $\widehat{R}(f)\subseteq R(f)$. Given $x\in \widehat{R}(f)$, we can find a sequence $(y_{k})_{k=1}^{\infty}$ of points $y_{k}\in C$ such that
  \begin{equation*}
    \lim_{k\to \infty}\frac{\rho(f(y_{k}),f(x))}{\rho(y_{k},x)}=1.
  \end{equation*}
  On the other hand, we have
  \begin{equation*}
    \frac{\rho(f(y_{k}),f(x))}{\rho(y_{k},x)}\leq \phi_{f}(\rho(y_{k},x))\leq 1
  \end{equation*}
  for all $k\in\mathbb{N}$, where $\phi_{f}\colon [0,\operatorname{diam}(C)]\to[0,1]$ is the function witnessing that $f$ is a Rakotch contraction. Letting $k\to \infty$, we deduce that $\lim_{k\to\infty}\phi_{f}(\rho(y_{k},x))=1$. Since $\phi_{f}$ is decreasing, we additionally have
  \begin{equation*}
    1=\lim_{k\to\infty}\phi_{f}(\rho(y_{k},x))\leq \phi_{f}(\limsup_{k\to\infty}\rho(y_{k},x)-\varepsilon)
  \end{equation*}
  for every $\varepsilon\in (0,\limsup_{k\to\infty}\rho(y_{k},x))$. To see this, evaluate the limit on the left-hand side along a subsequence which realises $\limsup_{k\to\infty}\rho(y_{k},x)$. Given that $\phi_{f}(t)<1$ for $t>0$, this leaves only the possibility $\limsup_{k\to\infty}\rho(y_{k},x)=0$. Hence $y_{k}\to x$ and so $x\in R(f)$. 
\end{proof}
Using Lemma~\ref{lemma:R=hat(R)} and Theorem~\ref{thm:typicalRakotch}, we now show that for typical $f\in\mathcal{M}$ in the $\mathrm{CAT}(\kappa)$ setting, the set $R(f)$ is residual in $C$.
\begin{theorem}\label{thm:star_shaped_local_lip}
  Let $(X,\rho_{X})$ be a $\mathrm{CAT}(\kappa)$ space with $\kappa\in\mathbb{R}$, 
  \begin{equation*}
    D_{\kappa}:=\begin{cases}
      \infty & \text{ if }\kappa\leq 0,\\
      \frac{\pi}{\sqrt{\kappa}} & \text{ if }\kappa >0,
    \end{cases}
  \end{equation*} 
  $C\subset X$ be a bounded closed set and $\mathcal{M}$ denote the space of nonexpansive mappings on $C$ defined by \eqref{eq:space_non_exp_mappings}. Let $x\in C$ be such that $C\subset B(x,D_{\kappa}/2)$ and for every $z\in C$ the metric segement $[z,x]$ connecting $z$ to $x$ is contained in $C$. Then there is a $\sigma$-porous subset $\mathcal{P}$ of $\mathcal{M}$ such that for every $f\in \mathcal{M}\setminus \mathcal{P}$ the set 
  \begin{equation*}
    R(f):=\left\{y\in C\colon \lip(f,y)=1\right\}
  \end{equation*}
  is a residual subset of $C$.
\end{theorem}
\begin{proof}
  The three key ingredients of the proof are Theorem~\ref{thm:typicalRakotch}, Lemma~\ref{lemma:R=hat(R)} and \cite[Theorem~3.7]{BDR2017}; the argument works in the same way as that establishing \cite[Corollary~3.9]{BDR2017} in \cite[Remark~3.8]{BDR2017}. We may assume that $C$ is not the singleton $\left\{x\right\}$, as otherwise $\mathcal{M}$ is also a singleton and the statement becomes trivial. Applying \cite[Theorem~3.7]{BDR2017} to $X$, $C_{X}=X$, $D_{X}=D_{\kappa}$, $Y=X$, $C_{Y}=C$ and $\mathcal{M}=\mathcal{M}(C_{X},C_{Y})$ (in the notation of \cite{BDR2017}), we obtain a $\sigma$-porous set  $\widetilde{\mathcal{N}}\subseteq \mathcal{M}$ such that for every $f\in \mathcal{M}\setminus\widetilde{\mathcal{N}}$ the set $\widehat{R}(f)$ is a residual subset of $C$. Moreover, for a Rakotch contractive $f\in\mathcal{M}\setminus \widetilde{\mathcal{N}}$ this set $\widehat{R}(f)$ coincides with $R(f)$, by Lemma~\ref{lemma:R=hat(R)}. Therefore, the desired $\sigma$-porous set $\mathcal{P}\subset \mathcal{M}$ may be taken as the union of $\widetilde{\mathcal{N}}$ and the $\sigma$-porous set $\mathcal{N}$ given by the conclusion of Theorem~\ref{thm:typicalRakotch}. 
\end{proof}
\section{Cases where the set of Banach contractions is not even dense}
We consider the case $X=\mathbb{S}^{n}$ of the $n$-dimensional sphere with the intrinsic metric. For a closed subset $C\subset X$, we define
\[
  \mathcal{M} := \{f\colon C\to C\colon \lip f \leq 1\}
\]
equipped with the metric
\[
  d_{\infty}(f,g) := \sup\{\rho_{\mathbb{S}^{n}}(f(x),g(x))\colon x\in C\}.
\]
It is not difficult to see that if $C$ is not convex, then the set of strict contractions need not be dense 
in~$\mathcal{M}$.
\begin{lemma}\label{lem:contractions-not-dense}
  We denote by $\mathbb{S}^{n}$ the $n$-dimensional sphere. Let $x\in \mathbb{S}^n$, $n\in\mathbb{N}\setminus\{0\}$, $0<r<R<\pi$ and $C=\overline{B(x,R)}\setminus B(x,r)$. The set of mappings which are strict contractions is not dense in $\mathcal{M}$.
\end{lemma}

\begin{proof}
  Let $\varepsilon>0$ be small enough and let $f:C\to C$ be the rotation with axis $x$ such that $\rho(y,f(y)) \geq \varepsilon$ for all $y\in C$. Assume that there is a strict contraction $g\in B(f, \tfrac{\varepsilon}{2})\subset \mathcal{M}$. Then $g$ has a fixed point $z\in C$. On the other hand, the inequalities \[\varepsilon \leq \rho(z,f(z)) = \rho(g(z),f(z))\leq \mathrm{d}_\infty (g,f) < \frac{\varepsilon}{2}\]
  hold, which is a contradiction.
\end{proof}

As a direct consequence of this Lemma we obtain the following Corollary.

\begin{corollary}
  Let $Y$ be a complete metric space such that a scaled version of $\mathbb{S}^n$ is contained in $Y$ as an isometric copy for some $n\in \mathbb{N}\setminus\{0\}$. Then there is a connected closed subset $C\subset Y$ such that the set of strict contractions is not dense in $\mathcal{M}$.
\end{corollary}

\begin{example}
  Here are three examples of spaces $Y$ to which this corollary can be applied.
  \begin{enumerate}
  \item The unit sphere of a real Hilbert space.
  \item The set
    \[
      Y := \{x\in \mathbb{R}^{n+1}\colon \|x\|_\infty \leq 2\;\text{and}\;\|x\|_2\geq 1\}
    \]
    with the intrinsic metric induced by $\mathbb{R}^{n+1}$.
  \item The set
    \[
      Y := \{x\in \mathbb{R}^{n+1}\colon 1 \leq \|x\|_2\leq 2\}
    \]
    with the intrinsic metric induced by $\mathbb{R}^{n+1}$.
  \end{enumerate}
\end{example}

Note that on the sphere, balls of radius larger than $\frac{\pi}{2}$ are not convex, since for a pair of points whose distance to the centre is larger than $\frac{\pi}{2}$ and lying on ``opposite sides'' of the centre, the unique metric segment connecting them passes through the complement of the ball. 

\begin{proposition}
  Let $C= \overline{B(x,R)}$ for some $x\in \mathbb{S}^n$ and $R\in [0,\pi[$. Then every continuous $f: C\to C$  
  has a fixed point.
\end{proposition}

\begin{proof}
  There is a homeomorphism $h:C\to K\subset \mathbb{R}^n$ with $K$ compact and convex; we may take, for example, the stereographic projection which maps the hemisphere onto a disc. The mapping $F: K\to K$ defined by $F = h \circ f \circ h^{-1}$ is a continuous self-mapping of a compact and convex set and has therefore a fixed point $z$ in $K$ by Brouwer's fixed point theorem. Then $x = h^{-1}(z) \in C$ is a fixed point of $f$.
\end{proof}

\begin{remark}
  In \cite[p. 70, Theorem 13.13.]{MR0461107} J.~H.~Wells and L.~R.~Williams proved that a nonexpansive mapping from a subset of the sphere into the sphere, whose range has diameter greater than $\pi$, or equivalently whose range is contained in no hemisphere, can be extended to an isometry on the whole sphere. Therefore strict contractions are not dense even in the case where $C=\overline{B(x,R)}$ for $R>\frac{\pi}{2}$.

  Combined with our result above, this shows that on $C$, although the set of strict contractions is not dense in the space of nonexpansive self-mappings, \emph{every} continuous self-mapping has a fixed point.
\end{remark}

\vskip2mm
\noindent\textbf{Acknowledgments.}
The first and the third author were supported by the Austrian Science Fund FWF (grant number P~32523-N). The second author was supported by the Austrian Science Fund FWF (grant number P~30902-N35).
The fourth author was partially supported by the Israel Science Foundation (Grant 820/17), by the Fund for the Promotion of Research at the Technion and by the Technion General Research Fund.


\vspace{8mm}
\noindent
Christian Bargetz\\
Department of Mathematics\\
University of Innsbruck\\
Technikerstraße 13,
6020 Innsbruck,
Austria\\
\texttt{christian.bargetz@uibk.ac.at}\\[3mm]
Michael Dymond\\
Department of Mathematics\\
University of Innsbruck\\
Technikerstraße 13,
6020 Innsbruck,
Austria\\
\texttt{michael.dymond@uibk.ac.at}\\[3mm]
Emir Medjic\\
Department of Mathematics\\
University of Innsbruck\\
Technikerstraße 13,
6020 Innsbruck,
Austria\\
\texttt{emir.medjic@uibk.ac.at}\\[3mm]
Simeon Reich\\
Department of Mathematics\\
The Technion---Israel Institute of Technology\\
32000 Haifa,
Israel\\
\texttt{sreich@math.technion.ac.il}

\end{document}